\documentclass[10pt,twoside]{siamltex}
\usepackage{mathrsfs}
\usepackage{amsfonts}
\usepackage{amsfonts,epsfig}

\newtheorem{example}[theorem]{Example}

\newcommand{\complex}{\mathbb{C}}

\begin{document}

%  Leave these commented lines here
% \input{elaheader-volx-xx.tex}
% \setcounter{page}{1}

% \renewcommand{\thefootnote}{\fnsymbol{footnote}}
% \renewcommand{\thefootnote}{\arabic{footnote}}
% \renewcommand{\theequation}{\thesection.\arabic{equation}}

\bibliographystyle{plain}
\title{
On parallel multisplitting block iterative methods for linear
systems arising in the numerical\\ solution of Euler
equations\thanks{Corresponding author: Cheng-yi Zhang, School of
Science, Xi'an Polytechnic University, Xi'an, Shaanxi, 710048, P.R.
China; Email: cyzhang08@126.com or zhangchengyi$_{-}$2004@163.com.}}
% Leave blank; editors will write the exact dates above

\author{
Cheng-yi Zhang\thanks{Institute of Information and system Science,
Xi'an Jiaotong University, Xi'an, Shaanxi, 710049, P.R. China;
School of Science, Xi'an Polytechnic University, Xi'an, Shaanxi,
710048, P.R. China. This work was partly supported by the Science
Foundation of the Education Department of Shaanxi Province of China\
(2013JK0593), the Scientific Research Foundation\ (BS1014) and the
Education Reform Foundation\ (2012JG40) of Xi'an Polytechnic
University, and the National Natural Science Foundations of China
(11201362 and 11271297).}
% Remember to put \and between any two authors
 \and Shuanghua Luo\thanks{School of Science, Xi'an
Polytechnic University, Xi'an, Shaanxi, 710048, P.R. China.}\and
Zongben Xu\thanks {Institute of Information and system Science and
School of Mathematics and Statistics, Xi'an Jiaotong University,
Xi'an, Shaanxi, 710049, P.R. China.} }
% Note that \footnotemark[3]} is used for the third author
% because of the same affiliation for the second and third authors.
% If the same affiliation is to be used for the first and second authors,
% \footnotemark[2] should be used instead of \thanks{} for the second author.

% Authors and running title to go on top of each page
\pagestyle{myheadings} \markboth{Cheng-yi Zhang, Shuanghua Luo and
Zongben Xu}{On parallel multisplitting block iterative methods for
linear systems} \maketitle
\begin{abstract}
The paper studies the convergence of some parallel multisplitting
block iterative methods for the solution of linear systems arising
in the numerical solution of Euler equations. Some sufficient
conditions for convergence are proposed. As special cases the
convergence of the parallel block generalized AOR (BGAOR), the
parallel block AOR (BAOR), the parallel block generalized SOR
(BGSOR), the parallel block SOR (BSOR), the extrapolated parallel
BAOR and the extrapolated parallel BSOR methods are presented.
Furthermore, the convergence of the parallel block iterative methods
for linear systems with special block tridiagonal matrices arising
in the numerical solution of Euler equations are discussed. Finally,
some examples are given to demonstrate the convergence results
obtained in this paper.
\end{abstract}

%%%%%%%%%%%%%%%%%%%%%%%%%%%%%%%%%%%%%%%%%%%%%%%%%%%%%%%%%%%%%
\section{Introduction} \label{intro-sec}
In this paper we consider the solution methods for the system of
$km$ linear equations
\begin{equation}\label{r1}
Ax=b, \end{equation} where $A=[A_{ij}]\in \complex^{km\times km}$ is
an $m\times m$ block matrix with all the blocks $A_{ij}\in
\complex^{k\times k}$, $b,\ x\in \complex^{km\times 1}$. The class
of systems arises not only in the numerical solution of $2D$ and
$3D$ Euler equations in fluid dynamics \cite{{E.J2},{P.S8},{R.N11}},
but also in the discretizations of PDEs associated to invariant tori
\cite{{L.J3},{L.G4}}.

Elsner and Mehrmann in \cite{{L.V5},{L.V6}} gave several convergence
results for some block iterative methods such as block Jacobi
method, block Gauss-Seidel method and block SOR method for the
solution of linear system (\ref{r1}) when the coefficient matrix $A$
is either generalized $M-$matrices (see
\cite{{L.V5},{L.V6},{R.N13}}) or consistently ordered $p-$cyclic
matrices (see \cite{{R.S.14}}). Later, Nabben \cite{{R.N11},{R.N12}}
established some further results on convergence of block iterative
methods for the solution of this class of linear systems with
conjugate generalized $H-$matrices (see \cite{{zhang21}}). For
example, he established convergence of the block Jacobi method, the
block Gauss-Seidel method, the block JOR-method and the block
SOR-method.

Recently, Zhang et al \cite{{zhang21}} further proposed several
convergence results for some block iterative methods including the
block Jacobi method, the block Gauss-Seidel method, the block SOR
method and the block AOR method for the solution of linear systems
when the coefficient matrices are generalized $H-$matrices.

In what follows we will introduce some iterative methods of the
system (\ref{r1}). Consider the following splitting of the
coefficient matrix $A$ of (\ref{r1}),
\begin{equation}\label{1r7}
A=D-L-U,
\end{equation}
where $D$ is nonsingular, $L$ and $U$ are not necessarily (block)
triangular in general. Assume that $det(D-\gamma L)\neq 0$. Then the
 (block) generalized accelerated overrelaxation (GAOR (BGAOR)) method is defined by
\begin{equation}\label{1r8}
x^{(i+1)}={\cal{L}}(\gamma,\omega)x^{(i)}+(D-\gamma L)^{-1}b,\ \ \ \
i=1,2,\cdots\cdots,
\end{equation}
where ${\cal{L}}(\gamma,\omega)=(D-\gamma
L)^{-1}[(1-\omega)D+(\omega-\gamma)L+\omega U]$ is the iteration
matrix of the method (\ref{1r8}). For $\omega=\gamma$, the (block)
generalized AOR method reduces to the (block) generalized SOR (GSOR
(BGSOR)) method. If the splitting (\ref{1r7}) is standard (block)
decomposition (i.e, $D$ is (block) diagonal and nonsingular, $L$ and
$U$ are strictly lower and strictly upper (block) triangular,
respectively), then the (block) generalized AOR method and the
(block) generalized SOR method reduce to the (block) AOR method and
the (block) SOR method, respectively. Furthermore, if the method
(\ref{1r8}) is the (block) AOR method and $\gamma=0$, then we obtain
the the (block) JOR method.

In this paper, we mainly discuss the convergence of parallel
multisplitting block iterative methods of linear system (\ref{r1}).
The parallel multisplitting iterative methods are investigated in
\cite{{F20},{N19},{OL16},{S17},{S18}}. Let us consider the block
case.

In order to solve the system (\ref{r1}) with parallel multisplitting
block iterative methods, the coefficient matrix $A=[A_{ij}]\in
\complex^{km\times km}$ is split into
\begin{equation}\label{r2}
A=M_s-N_s,\ \ \ \ s=1,2,\cdots,r
\end{equation}
by means of the following block matrices $M_s=[M_{ij}]$ with
\begin{equation}\label{r3}
 M_{ij}= \left\{
\begin{array}{cc}
A_{ij},\ \ & \ \ {\rm if}\ \ \ (i,j)\in Q_s\ and\ i=j\in N\\
0, \ \ & \ \ {\rm if} \ \ (i,j)\notin\ Q_s,\ i\neq j
\end{array}\right.
\end{equation}
and $N_s=[N_{ij}]$ with
 \begin{equation}\label{r4}
N_{ij}= \left\{
\begin{array}{cc}
0,\ \ & \ \ {\rm if}\ \ \ (i,j)\in Q_s\ and\ i=j\in N\\
-A_{ij}, \ \ & \ \ {\rm if} \ \ (i,j)\notin\ Q_s,\ i\neq j
\end{array}\right..
\end{equation}
Here $Q_s\subset P(m)=\{(i,j)\ |\ i,j\in N=\{1,2,\cdots,m\},\ i\neq
j\}$ and each $M_s$ is nonsingular for $s=1,2,\cdots,r$. The
splittng (\ref{r2}) is called a multisplitting of the matrix $A$ and
is denoted by $(M_s,N_s,E_s)_{s=1}^{r}$. Here,
$E_s=diag(e_s^1I_k,e_s^2I_k,\cdots,e_s^mI_k)$ is a $km\times km$
nonnegative diagonal matrix for $s=1,2,\cdots,r$ and
$\sum_{s=1}^{r}E_s=I$, the $km\times km$ identity matrix. It follows
that a parallel multisplitting block iterative form of (\ref{r1})
can be described as follows:
\begin{equation}\label{r5}
x^{(i+1)}=\sum\limits_{s=1}^{r}E_sM_s^{-1}N_sx^{(i)}+\sum\limits_{s=1}^{r}E_sM_s^{-1}b,\
\ \ \ i=1,2,\cdots\cdots
\end{equation}
With $T=\sum_{s=1}^{r}E_sM_s^{-1}N_s$ and calling $T$ the iteration
matrix of the method (\ref{r5}), eq. (\ref{r5}) can be changed into
the following equations:
\begin{equation}\label{r5'}
 \begin{array}{lll}
 x^{(i+1)}&=&\sum\limits_{s=1}^{r}E_sy^{(i)}_s,\ \ \ i=1,2,\cdots\cdots,\\
y^{(i)}_s&=&M_s^{-1}N_sx^{(i)}+M_s^{-1}b \ \ \ s=1,2,\cdots,r.
\end{array}
\end{equation}
Eq. (\ref{r5'}) shows that this multisplitting method has a natural
parallelism, since the calculations of $y^{(i)}_s$ for various
values of $s$ are independent and may therefore be performed in
parallel. Moreover, the $j$th component of $y^{(i)}_s$ need not be
computed if the corresponding diagonal entry of $E_s$ is zero. This
may result in considerable savings of computational time.

If $r=1$, then the multisplitting (\ref{r2}) turns into a single
splitting
\begin{equation}\label{r2'}
A=M_1-N_1,
\end{equation}
and the corresponding block iterative method is a general block
iterative method.

An extrapolated parallel iterative method with a positive
extrapolation parameter $\tau$ is considered in \cite{{S17}} and
\cite{{F20}}. The following gives the extrapolated parallel block
iterative method by the block iteration
\begin{equation}\label{r6}
x^{(i+1)}=\tau\sum\limits_{s=1}^{r}E_sM_s^{-1}(N_sx^{(i)}+b)+(1-\tau)x^{(i)},\
\ \ \ i=1,2,\cdots\cdots
\end{equation}
Its iteration matrix is defined by
$$T(\tau)=\tau\sum\limits_{s=1}^{r}E_sM_s^{-1}N_s+(1-\tau)I.$$

In \cite{{S17}} and \cite{{S18}}, the parallel generalized AOR
(GAOR), block AOR (BAOR) and AOR methods are defined. Let
\begin{equation}\label{r7}
A=D_s-L_s-U_s,\ \ \ \ s=1,2,\cdots,r
\end{equation}
where $D_s\in C^{km\times km}$ is a nonsingular block matrix,
$L_k\in \complex^{km\times km}$ and $U_k\in \complex^{km\times km}$
are not necessarily block triangular in general. Assume that
$det(D_s-\gamma_sL_s)\neq 0,\ \ s=1,2,\cdots,r$. Then the parallel
block GAOR (BGAOR) method is defined by
\begin{equation}\label{r8}
x^{(i+1)}={\cal{L}}(\Gamma,\Omega)x^{(i)}+\sum\limits_{s=1}^{r}E_s(D_s-\gamma_sL_s)^{-1}b,\
\ \ \ i=1,2,\cdots\cdots,
\end{equation}
where
 \begin{equation}\label{r9}
 \begin{array}{lll}
&&{\cal{L}}(\Gamma,\Omega)=\sum\limits_{s=1}^{r}E_s(D_s-\gamma_sL_s)^{-1}[(1-\omega_s)D_s+(\omega_s-\gamma_s)L_s+\omega_sU_s],\\
&& \Gamma=(\gamma_1,\gamma_2,\cdots,\gamma_r),\ \ \ \
\Omega=(\omega_1,\omega_2,\cdots,\omega_r).
\end{array}
\end{equation}
This method may be achieved by the multisplitting (\ref{r2}) with
\begin{equation}\label{r10}
\begin{array}{lll}
{M}_s&=&\displaystyle\frac{1}{\omega_s}(D_s-\gamma_sL_s),\\
{N}_s&=&\displaystyle\frac{1}{\omega_s}[(1-\omega_s)D_s+(\omega_s-\gamma_s)L_s+\omega_sU_s],
\ \ s=1,2,\cdots,r.
\end{array}
\end{equation}

The parallel BGAOR method reduces to the parallel BGSOR (parallel
block generalized SOR) method if the parameter pairs
$(\gamma_s,\omega_s)$ turn into $(\omega_s,\omega_s)$ for
$s=1,2,\cdots,r$ and the parallel BGGS (parallel block generalized
Gauss-Seidel) method if the parameter pairs $(\gamma_s,\omega_s)$
turn into $(\omega_s,\omega_s)$ with $\omega_s=1$ for
$s=1,2,\cdots,r.$ We denote by ${\cal{L}}(\Omega)$ and
${\cal{L}}_{PBGGS}$ the iteration matrices of the parallel BGSOR and
the parallel BGGS methods, respectively.

If the decompositions in (\ref{r7}) are the usual block
decompositions, i.e., $D_s\in \complex^{km\times km}$ is a
nonsingular block diagonal part of $A$, $L_k\in \complex^{km\times
km}$ and $U_k\in \complex^{km\times km}$ are strictly lower and
upper block triangular matrices, respectively, then the parallel
BGAOR and the parallel BGSOR methods reduce to the parallel block
AOR (BAOR) and the parallel block SOR (BSOR) methods, respectively.
Lastly, we denote the iteration matrices of the extrapolated BGAOR
and BGSOR methods by ${\cal{L}}(\Gamma,\Omega,\tau)$ and
${\cal{L}}(\Omega,\tau)$, respectively.

This paper is organized as follows. Some notations and preliminary
results about generalized $H-$matrices are given in Section 2. The
convergence results of parallel block iterative methods for linear
systems with generalized $H-$matrices are established in Section 3.
In what follows, the convergence properties of parallel block
iterative methods for linear systems with special block tridiagonal
matrices arising in special cases from the computations of partial
differential equations are discussed in Section 4 and some examples
are given in Section 5 to illustrate the convergence results
obtained in this paper. Finally, conclusions are given in Section 6.
%%%%%%%%%%%%%%%%%%%%%%%%%%%%%%%%%%%%%%%
\section{Preliminaries} \label{prelimi-sec}
In this section we give some notions and preliminary results about
special matrices that are used in this paper. We denote by
$\complex^{n\times n}\ (\mathbb{R}^{n\times n})$ the set of all
$n\times n$ complex (real) matrices; $\complex^n$ the set of all
$n-$dimensional complex vectors; $\mathbb{R}_+^n$ the set of
positive vectors in $\mathbb{R}^n$; $A^T$ the transpose of $A$;
$A^H$ the conjugate transpose of $A$; $\rho(A)$ the spectral radius
of $A$; $Re(z)$ the real part of $z$.

\begin{definition}  {\rm (see \cite{{R.A9}})}
A matrix $A\in \complex^{n\times n}$ is called Hermitian if $A^H=A$;
a Hermitian matrix $A\in \complex^{n\times n}$ is called Hermitian
positive definite if $x^HAx>0$ for all $0\neq x\in \complex^n$ and
Hermitian semipositive definite if $x^HAx\geq 0$ for all $x\in
\complex^n$. A matrix $A\in \complex^{n\times n}$ is called positive
definite if $Re(x^HAx)>0$ for all $0\neq x\in \complex^n$ and
semipositive definite if $Re(x^HAx)\geq 0$ for all $x\in
\complex^n$.
\end{definition}

By $A>0$ and $A\geq 0$ we denote that $A$ is (Hermitian) positive
definite and (Hermitian) semipositive definite. Analogously we write
$A<0$ if $-A>0$ and $A\leq 0$ if $-A\geq 0$. Furthermore, for $A,\
B\in \complex^{n\times n}$, we write $A>B$ and $A\geq B$ if $A-B>0$
and $A-B\geq0$.

\begin{definition}
Let $A=(a_{ij})\in \complex^{n\times n}$. If $A$ is Hermitian, then
$|A|\in \complex^{n\times n}$ is defined as $|A|:=\sqrt{AA}.$
\end{definition}

\begin{definition}  {\rm (see \cite{{L.V5},{R.N11}})}
\begin{enumerate}
\item $Z_m^k=\{A=[A_{ij}]\in \complex^{km\times
km}\ |\ A_{ij}\in \complex^{k\times k}\ is\ Hermitian\ for\ all\
i,j\in N=\{1,2,\cdots,m\}\ and$ $A_{ij}\leq 0\ for\ all\ i\neq j,\
i,j\in N\}$;
\item $\widehat{Z}_m^k=\{A=[A_{ij}]\in Z_m^k\ |\ A_{ii}>0,\
i\in N\}$;
\item $M_m^k=\{A\in \widehat{Z}_m^k\ |\ there\ exists\ u\in \mathbb{R}_+^m\ such\ that\ \sum_{j=1}^{m}u_jA_{ij}>0\ for\ all\
i\in N\}$, where $\mathbb{R}_+^m$ denotes all positive vectors in
$\mathbb{R}^m$, and A matrix $A\in \widehat{Z}_m^k$ is called a
generalized $M-$matrix if $A\in M_m^k$;
\item $D_m^k=\{A=[A_{ij}]\in \complex^{km\times
km}\ |\ A_{ij}\in \complex^{k\times k}\ is\ Hermitian\ for\ all\
i,j\in N$ $and\ A_{ii}>0\ for\ all\ i\in N\}$;
\item $H_m^k=\{A\in D_m^k\ |\ \mu(A)\in M_m^k\}$, where $\mu(A)=[M_{ij}]\in \complex^{mk\times mk}$ is the block comparison matrix of $A$ and is defined as
$$
M_{ij}:=\left\{
\begin{array}{cc}
|A_{ii}|,\ \ & \ \ {\rm if}\ \ i=j\\
-|A_{ij}|, \ \ & \ \ {\rm if} \ \ i\neq j
\end{array} \right.,$$  and A
matrix $A\in D_m^k$ is called a generalized $H-$matrix if $A\in
H_m^k$.
\end{enumerate}
\end{definition}

\section{Main results}\label{main-sec}
In this section we discuss the convergence of parallel
mulitisplitting block iterative methods when the coefficient
matrices are generalized $H-$matrices. The following lemmas will be
used in this section.

\begin{lemma}
Let $A=(a_{ij})\in \complex^{n\times n}$ with a multisplitting
$(M_s,N_s,E_s)_{s=1}^{r}$, and let $T=\sum_{s=1}^{r}E_sM_s^{-1}N_s$
and $\hat{A}=\hat{M}-\hat{N}$, where

\begin{equation}\label{r12}
\begin{array}{llll}
\hat{M}= \left[
 \begin{array}{cccc}
 \ M_1 & 0 & \cdots & 0  \\
 \ 0 & M_2 & \cdots & 0  \\
 \ \vdots & \vdots & \ddots & \vdots  \\
 \ 0 & 0 & \cdots & M_r  \\
 \end{array}
 \right],\
 \hat{N}= \left[
 \begin{array}{cccc}
 \ N_1E_1 & N_1E_2 & \cdots & N_1E_r  \\
 \ N_2 E_1 & N_2E_2 & \cdots & N_2E_r  \\
 \ \vdots & \vdots & \ddots & \vdots  \\
 \ N_r E_1 & N_rE_2 & \cdots & N_rE_r  \\
 \end{array}
 \right].
\end{array}
\end{equation}
Then $\rho(T)=\rho({\hat{M}}^{-1}{\hat{N}})$, where $\rho(T)$
denotes the spectral radius of the matrix $T$.
\end{lemma}

\begin{proof} \begin{equation}\label{r13}
\begin{array}{llll}
\rho(T)&=& \rho(\sum\limits_{s=1}^{r}E_sM_s^{-1}N_s)\\
&=&\rho\left(\left[
 \begin{array}{cccc}
 \ E_1 & E_2 & \cdots & E_r  \\
 \ 0 & 0 & \cdots & 0  \\
 \ \vdots & \vdots & \ddots & \vdots  \\
 \ 0 & 0 & \cdots & 0  \\
 \end{array}
 \right]\left[
 \begin{array}{cccc}
 \ M_1^{-1}N_1 & 0 & \cdots & 0  \\
 \ M_2^{-1}N_2 & 0 & \cdots & 0  \\
 \ \vdots & \vdots & \ddots & \vdots  \\
 \ M_r^{-1}N_r & 0 & \cdots & 0  \\
 \end{array}
 \right]\right)\\
&=&\rho\left(\left[
 \begin{array}{cccc}
 \ M_1^{-1}N_1 & 0 & \cdots & 0  \\
 \ M_2^{-1}N_2 & 0 & \cdots & 0  \\
 \ \vdots & \vdots & \ddots & \vdots  \\
 \ M_r^{-1}N_r & 0 & \cdots & 0  \\
 \end{array}
 \right]\left[
 \begin{array}{cccc}
 \ E_1 & E_2 & \cdots & E_r  \\
 \ 0 & 0 & \cdots & 0  \\
 \ \vdots & \vdots & \ddots & \vdots  \\
 \ 0 & 0 & \cdots & 0  \\
 \end{array}
 \right]\right)\end{array}
\end{equation}
\begin{eqnarray*}
 &=&\rho\left(\left[
 \begin{array}{cccc}
 \ M_1^{-1}N_1E_1 & M_1^{-1}N_1E_2 & \cdots & M_1^{-1}N_1E_r  \\
 \ M_2^{-1}N_2 E_1 & M_2^{-1}N_2E_2 & \cdots & M_2^{-1}N_2E_r  \\
 \ \vdots & \vdots & \ddots & \vdots  \\
 \ M_r^{-1}N_r E_1 & M_r^{-1}N_rE_2 & \cdots & M_r^{-1}N_rE_r  \\
 \end{array}
 \right]\right)\\
 &=&\rho(\hat{M}^{-1}\hat{N}),
\end{eqnarray*}
where $\hat{M}$ and $\hat{N}$ are defined as in (\ref{r12}). This
completes the proof.
\end{proof}

\begin{lemma}{\rm (see \cite{{zhang21}})}
\label{lem} Let $A=[A_{ij}]\in H_m^k$ with a splitting $A=M_1-N_1$
as in {\rm (\ref{r2'})}. Then $\rho({M_1}^{-1}{N_1})<1$.
\end{lemma}

\begin{theorem}
Let $A=[A_{ij}]\in H_m^k$ with a multisplitting
$(M_s,N_s,E_s)_{s=1}^{r}$. Then the parallel multisplitting block
iterative method {\rm (\ref{r5})} converges to the unique solution
of {\rm(\ref{r1})} for any choice of the initial guess $x^{(0)}$.
\end{theorem}

\begin{proof} We only prove that $\rho(T)<1$. Lemma 3.1 shows that $\rho(T)=\rho({\hat{M}}^{-1}{\hat{N}})$,
where $\hat{M}$ and $\hat{N}$ are defined as in (\ref{r12}). Since
$A\in H_m^k$ indicates $\mu(A)\in M_m^k,$ it follows from Definition
2.3 that there exists a positive diagonal matrix
$F=diag(f_1I_k,f_2I_k,\cdots,f_mI_k)$, where $I_k$ is the $k\times
k$ identity matrix, such that $AF$ satisfies
\begin{equation}\label{r14}
f_i|A_{ii}|-\sum\limits_{j=1,j\neq i}^{m}|A_{ij}|f_j>0,
\end{equation}
for all $i\in N$. Note that $(M_s,N_s,E_s)_{s=1}^{r}$ is a
multisplitting of $A$, $E_s=diag(e_s^1I_k,\cdots,e_s^mI_k)$ is a
$km\times km$ nonnegative diagonal matrix for $s=1,2,\cdots,r$ and
$\sum_{s=1}^{r}E_s=I$, the $km\times km$ identity matrix. Then we
have
\begin{equation}\label{r15}
\sum\limits_{s=1}^{r}e_s^i=1,\ i=1,2,\cdots,m\ \ and\ \ e_s^i\geq
0,\ s=1,2,\cdots,r.
\end{equation}
As a result, $A=M_s-N_s=M_s-N_s\sum_{s=1}^{r}E_s\in H_m^k$
satisfying (\ref{r14}) for all $s=1,2,\cdots,r$. Following
(\ref{r14}) and (\ref{r15}), we have that for $s=1,2,\cdots,r,$
\begin{equation}\label{r16}
\begin{array}{llll}
&&\left[f_i|A_{ii}|-\sum\limits_{(i,j)\in
Q_s}|A_{ij}|f_j\right]-\sum\limits_{s=1}^{r}\left[\sum\limits_{(i,j)\underline{\in}Q_s;j\neq i}|A_{ij}|f_j\right]e_s^i\\
&&=f_i|A_{ii}|-\sum\limits_{(i,j)\in Q_s}|A_{ij}|f_j
-\sum\limits_{(i,j)\underline{\in}Q_s;j\neq i}(\sum\limits_{s=1}^{r}|A_{ij}|e_s^i)f_j\\
&&=f_i|A_{ii}|-\left[\sum\limits_{(i,j)\in
Q_s}|A_{ij}|f_j+\sum\limits_{
(i,j)\underline{\in}Q_s;j\neq i}|A_{ij}|f_j\right]\\
&&>0,\ \ \ \ i=1,2,\cdots,m.
\end{array}
\end{equation}
Thus, there exists a positive diagonal matrix
$\hat{F}=diag(F,F,\cdots,F)$ such that $\hat{A}\hat{F}$ satisfies
(\ref{r16}) for $i=1,2,\cdots,m\ and\ s=1,2,\cdots,r,$ which shows
that $\hat{A}\in H_{rm}^k$. From (\ref{r12}), we know that
$\hat{A}=\hat{M}-\hat{N}$ is a splitting  as in (\ref{r2'}). It then
follows from Lemma 3.2 that $\rho(T)=\rho({M_Q}^{-1}{N_Q})<1$ which
completes the proof.
\end{proof}

\begin{theorem} Let $A=[A_{ij}]\in H_m^k$ with a
multisplitting $(M_s,N_s,E_s)_{s=1}^{r}$. Then the extrapolated
parallel multisplitting block iterative method {\rm (\ref{r6})}
converges to the unique solution of {\rm(\ref{r1})} for any choice
of the initial guess $x^{(0)}$, provided $\tau\in (0,2/(1+\rho))$,
where $\rho=\rho(T)$ and $T$ is the iteration matrix of the method
(\ref{r5}).
\end{theorem}

\begin{proof}
Since the iteration matrix of the extrapolated parallel
multisplitting block iterative method is
$$T(\tau)=\tau\sum\limits_{s=1}^{r}E_sM_s^{-1}N_s+(1-\tau)I=\tau
T+(1-\tau)I,$$ we have $\rho(T(\tau))=\rho(\tau T+(1-\tau)I)\leq
\tau\rho(T)+|1-\tau|$. Theorem 3.3 implies that $\rho(T)<1$. As a
result, $\rho(T(\tau))\leq \tau\rho(T)+|1-\tau|<1$ for all $\tau\in
(0,2/(1+\rho))$. Thus, the extrapolated parallel multisplitting
block iterative method converges to the unique solution of
{\rm(\ref{r1})} for any choice of the initial guess $x^{(0)}$. This
completes the proof.
\end{proof}

In what follows, we consider convergence of the parallel BGAOR
iterative method of the system (\ref{r1}).

\begin{theorem} Let $A=[A_{ij}]\in H_m^k$ with a
multisplitting {\rm (\ref{r7})}. If\ $0\leq
\gamma_s\leq\omega_s\leq1$ and $0<\omega_s$ for $s=1,2,\cdots,r,$
then the parallel BGAOR iterative method {\rm (\ref{r8})} converges
to the unique solution of {\rm(\ref{r1})} for any choice of the
initial guess $x^{(0)}$.
\end{theorem}

\begin{proof}
Since the parallel BGAOR iterative method {\rm (\ref{r8})} is
induced by the multisplitting $(M_s,N_s,E_s)_{s=1}^{r}$ defined in
(\ref{r2}) with
\begin{equation}\label{add1}
\begin{array}{lll}
{M}_s&=&\displaystyle\frac{1}{\omega_s}(D_s-\gamma_sL_s),\\
{N}_s&=&\displaystyle\frac{1}{\omega_s}[(1-\omega_s)D_s+(\omega_s-\gamma_s)L_s+\omega_sU_s],
\ \ s=1,2,\cdots,r,
\end{array}
\end{equation}
it follows from Lemma 3.1 that
$\rho({\cal{L}}(\Gamma,\Omega))=\rho(\sum_{s=1}^{r}E_sM_s^{-1}N_s)=\rho({\hat{M}}^{-1}{\hat{N}})$,
where $\hat{M}$ and $\hat{N}$ are defined as (\ref{r12}). Following,
we will prove that $\hat{A}={\hat{M}}-{\hat{N}}$ is a generalized
$H-$matrix. Let $R_s,S_s,T_s\subset P(m)=\{(i,j)\ |\ i,j\in
N=\{1,2,\cdots,m\},\ i\neq j\}$, $R_s\cap S_s=R_s\cap T_s=T_s\cap
S_s=\emptyset$ and $R_s\cup S_s\cup T_s=P(m)$. Then for
$s=1,2,\cdots,r,$ $D_s=[D_{ij}]\in \complex^{km\times km}$,
$L_s=[L_{ij}]\in \complex^{km\times km}$ and $U_s=[U_{ij}]\in
\complex^{km\times km}$ in (\ref{add1}) are defined by
\begin{equation}\label{xr3}
\begin{array}{llllll}
 D_{ij}&=& \left\{
\begin{array}{cc}
A_{ij},\ \ & \ \ \ \ (i,j)\in R_s\ and\ i=j\in N\\
0, \ & (i,j)\overline{\in}\ R_s,\ i\neq j  \ \ \ \ \ \
\end{array}\right.\\
L_{ij}&=& \left\{
\begin{array}{cc}
A_{ij},\ \ & \ \ \ \ \ (i,j)\in S_s\\
0, \ \ & \ \ \ \ (i,j)\overline{\in}\ S_s,
\end{array}\right.,\\
U_{ij}&=& \left\{
\begin{array}{cc}
A_{ij},\ \ & \ \ \ \ \ (i,j)\in T_s\\
0, \ \ & \ \  \ \ (i,j)\overline{\in}\ T_s
\end{array}\right..
\end{array}
\end{equation}

Since $A\in H_m^k$ indicates $\mu(A)\in M_m^k,$ Definition 2.3 shows
that there exists a positive diagonal matrix
$F=diag(f_1I_k,f_2I_k,\cdots,f_mI_k)$, where $I_k$ is the $k\times
k$ identity matrix, such that $AF$ satisfies
\begin{equation}\label{xr14}
f_i|A_{ii}|-\sum\limits_{j=1,j\neq i}^{m}|A_{ij}|f_j>0,
\end{equation}
for all $i\in N$. Note that $(M_s,N_s,E_s)_{s=1}^{r}$ is a
multisplitting of $A$, $E_s=diag(e_s^1I_k,\cdots,e_s^mI_k)$ is a
$km\times km$ nonnegative diagonal matrix for $s=1,2,\cdots,r$ and
$\sum_{s=1}^{r}E_s=I$, the $km\times km$ identity matrix. Then we
have
\begin{equation}\label{xr15}
\sum\limits_{s=1}^{r}e_s^i=1,\ i=1,2,\cdots,m\ \ and\ \ e_s^i\geq
0,\ s=1,2,\cdots,r.
\end{equation}
As a result, $A=M_s-N_s=M_s-N_s\sum_{s=1}^{r}E_s\in H_m^k$
satisfying (\ref{xr14}) for all $s=1,2,\cdots,r$. Let
$\hat{A}=[\hat{A}_{ij}]\in C_{rm}^k$. Since $0\leq
\gamma_s\leq\omega_s\leq1$ and $0<\omega_s$ for $s=1,2,\cdots,r,$ it
follows from (\ref{xr14}) and (\ref{xr15}) that
\begin{equation}\label{xr16}
\begin{array}{llll}
&&f_i|\hat{A}_{ii}|-\sum\limits_{s=1}^r\sum\limits_{j=1,j\neq
i}^m|\hat{A}_{i,(s-1)m+j}|f_j\\&\geq&\left[\left(f_i|A_{ii}|-\sum\limits_{(i,j)\in
R_s;j\neq i}|A_{ij}|f_j\right)-\gamma_s\sum\limits_{(i,j)\in
S_s}|A_{ij}|f_j\right]\\&&~~~~-\sum\limits_{s=1}^{r}\left[(1-\omega_s)\left(f_i|A_{ii}|-\sum\limits_{(i,j)\in
R_s;j\neq
i}|A_{ij}|f_j\right)\right.\\&&~~~~~~~\left.+(\omega_s-\gamma_s)\sum\limits_{(i,j)\in
S_s}|A_{ij}|f_j+\omega_s\sum\limits_{
(i,j)\underline{\in}T_s}|A_{ij}|f_j\right]e_s^i\\
&=&\left[\left(f_i|A_{ii}|-\sum\limits_{(i,j)\in R_s;j\neq
i}|A_{ij}|f_j\right)-\gamma_s\sum\limits_{(i,j)\in
S_s}|A_{ij}|f_j\right]\\&&~~~~-\left[(1-\omega_s)\left(f_i|A_{ii}|-\sum\limits_{(i,j)\in
R_s;j\neq
i}|A_{ij}|f_j\right)\right.\\&&~~~~~~~\left.+(\omega_s-\gamma_s)\sum\limits_{(i,j)\in
S_s}|A_{ij}|f_j+\omega_s\sum\limits_{
(i,j)\underline{\in}T_s}|A_{ij}|f_j\right]\\
&=&\omega_s\left[f_i|A_{ii}|-\sum\limits_{(i,j)\in R_s;j\neq
i}|A_{ij}|f_j-\sum\limits_{(i,j)\in S_s}|A_{ij}|f_j-\sum\limits_{
(i,j)\underline{\in}T_s}|A_{ij}|f_j\right]\\
&=&f_i|A_{ii}|-\sum\limits_{j=1,j\neq i}^{m}|A_{ij}|f_j\\
&>&0,\ \ \ \ i=1,2,\cdots,m; s=1,2,\cdots,r.
\end{array}
\end{equation} Therefore,
there exists a positive diagonal matrix $\hat{F}=diag(F,F,\cdots,F)$
such that $\hat{A}\hat{F}$ satisfies (\ref{xr16}) for
$i=1,2,\cdots,m\ and\ s=1,2,\cdots,r,$ which shows that $\hat{A}\in
H_{rm}^k$. (\ref{r12}) shows that $\hat{A}=\hat{M}-\hat{N}$ is a
splitting as in (\ref{r2'}). It then follows from Lemma 3.2 that
$\rho({\cal{L}}(\Gamma,\Omega)=\rho(\sum_{s=1}^{r}E_sM_s^{-1}N_s)=\rho({\hat{M}}^{-1}{\hat{N}})<1$
which completes the proof.
\end{proof}

It is easy to obtain immediately the following corollaries from
Theorem 3.5.

\begin{corollary} Let $A=[A_{ij}]\in H_m^k$ with a
multisplitting {\rm (\ref{r7})}. If $0\leq
\gamma_s\leq\omega_s\leq1$ and $0<\omega_s$ for $s=1,2,\cdots,r,$
then the parallel BAOR iterative method converges to the unique
solution of {\rm(\ref{r1})} for any choice of the initial guess
$x^{(0)}$.
\end{corollary}

\begin{corollary} Let $A=[A_{ij}]\in H_m^k$ with a
multisplitting {\rm (\ref{r7})}. If $0<\omega_s\leq1$ for
$s=1,2,\cdots,r,$ then the parallel BGSOR and BSOR iterative method
 converges to the unique solution of {\rm(\ref{r1})}
for any choice of the initial guess $x^{(0)}$.
\end{corollary}

\begin{theorem} Let $A=[A_{ij}]\in H_m^k$ with a
multisplitting {\rm (\ref{r7})}. If $0\leq
\gamma_s\leq\omega_s\leq1$ and $0<\omega_s$ for $s=1,2,\cdots,r,$
then the extrapolated parallel BGAOR iterative method converges to
the unique solution of {\rm(\ref{r1})} for any choice of the initial
guess $x^{(0)}$.
\end{theorem}

\begin{proof}
Similar to the proof of Theorem 3.4, it is easy to obtain the proof
coming from Theorem 3.5.
\end{proof}

\begin{corollary} Let $A=[A_{ij}]\in H_m^k$ with a
multisplitting {\rm (\ref{r7})}. If $0<\omega_s\leq1$ for
$s=1,2,\cdots,r,$ then the extrapolated parallel BGSOR iterative
method converges to the unique solution of {\rm(\ref{r1})} for any
choice of the initial guess $x^{(0)}$.
\end{corollary}

%%%%%%%%%%%%%%%%%%%%%%%%%%%%%%%%%%%%%%%

\section{Applications to special cases from the solution
of partial differential equations} \label{applications-sec} In this
section, we will discuss the convergence of matrices arising in the
numerical solution of some special partial differential equations
such as the Euler equation \cite{{P.S8}}, the Navier-Stokes equation
\cite{{E.J2}}, elliptic equations \cite{{R.S.14}} and so on. These
matrices have the following form
\begin{equation}\label{wh1}
M:=\left[
 \begin{array}{ccccccccc}
 \ T & S_1 &  &  \\
 \ S_2 & T & \ddots &  \\
 \   & \ddots & \ddots & S_1\\
 \  &  & S_2 & T
 \end{array}
 \right]\in \complex^{prk\times prk},
\end{equation}
where $T_1,\ S_1,\ S_2\in \complex^{rk\times rk}$ are defined by
\begin{equation}\label{wh2}
\begin{array}{llll}
T=\left[
 \begin{array}{ccccccccc}
 \ C & -A^- &  &  \\
 \ -A^+ & C & \ddots &  \\
 \   & \ddots & \ddots & -A^-\\
 \  &  & -A^+ & C
 \end{array}
 \right],
 \end{array}
\end{equation}
\begin{equation}\label{wh2'}
\begin{array}{llll}
S_1&=&\left[
 \begin{array}{ccccccccc}
 \  -B^- &  &  \\
 \   & \ddots &   \\
 \   &  & -B^-
 \end{array}
 \right],\ S_2=\left[
 \begin{array}{ccccccccc}
 \  -B^+ &  &  \\
 \   & \ddots &   \\
 \   &  & -B^+
 \end{array}
 \right].
 \end{array}
\end{equation}
Here $A=A^+-A^-\in \complex^{k\times k}$ and $B=B^+-B^-\in
\complex^{k\times k}$ are decompositions of Hermitian (indefinite)
matrices $A$, $B$ into positive semidefinite parts $A^+$, $B^+$ and
negative semidefinite parts $-A^-$, $-B^-$, while
$C=A^++A^-+B^++B^-$. Furthermore, $N(A)\cap N(B)=\emptyset$, where
$N(A)=\{x\in \complex^n\ |\ Ax=0\}$ is the right null space of the
matrix $A$.

With $T=M_s-N_s,\ s=1,2,\cdots,t,$ where $M_s$ and $N_s$ are defined
by (\ref{r3}) and (\ref{r4}), one has the splitting
\begin{equation}\label{yuan0}
M=P_s-Q_s,\ s=1,2,\cdots,t,
\end{equation}
 where
\begin{equation}\label{yuan1}
P_s=diag(M_s,M_s,\cdots, M_s)\in \complex^{prk\times prk},
\end{equation}
and
\begin{equation}\label{yuan2}Q_s=\left[
 \begin{array}{ccccccccc}
 \ N_s & -S_1 &  &  \\
 \ -S_2 & N_s & \ddots &  \\
 \   & \ddots & \ddots & -S_1\\
 \  &  & -S_2 & N_s
 \end{array}
 \right]\in \complex^{prk\times prk}.
\end{equation}
Let
\begin{equation}\label{yuan3}
T=D'_s-L'_s-U'_s,\ \ \ \ s=1,2,\cdots,t
\end{equation}
be as in (\ref{r7}). Then the matrix $M$ can be written as
\begin{equation}\label{yuan4}
M=D_s-L_s-U_s,\ \ \ \ s=1,2,\cdots,t,
\end{equation}
where
\begin{equation}\label{yuan5}
\begin{array}{lll}
D_s&=&diag(D'_s,D'_s,\cdots,D'_s)\in \complex^{prk\times prk},\\
L_s&=&\left[
 \begin{array}{ccccccccc}
 \ L'_s &  &  &  \\
 \ -S_2 & L'_s &  &  \\
 \   & \ddots & \ddots & \\
 \  &  & -S_2 & L'_s
 \end{array}
 \right]\in \complex^{prk\times prk},
 \end{array}
\end{equation}
and \begin{equation}\label{yuan6} U_k=\left[
 \begin{array}{ccccccccc}
 \ U'_s & -S_1 &  &  \\
 \  & U'_s & \ddots &  \\
 \   &  & \ddots & -S_1\\
 \  &  &  & U'_s
 \end{array}
 \right]\in \complex^{prk\times prk}.
\end{equation}

Based on the splittings (\ref{yuan0}) and (\ref{yuan4}), this
section will establish some convergence results for the parallel
multisplitting block iterative method and the parallel
multisplitting block GAOR (AOR) method, respectively.

\begin{theorem} Let $M$ be as in {\rm (\ref{wh1})}, {\rm (\ref{wh2})} and {\rm (\ref{wh2'})}.
For the splitting {\rm (\ref{yuan0})} of $M$, the parallel
multisplitting block iterative method {\rm (\ref{r5})} converges to
the unique solution of {\rm(\ref{r1})} for any choice of the initial
guess $x^{(0)}$.
\end{theorem}

\begin{proof}
According to Theorem 6.1 in \cite{{R.N11}}, we have $M+M^H\in
M^k_{pr}$. It is easy to obtain $M\in H^k_{pr}$ from Lemma 3.1 in
\cite{{H.S10}}. It follows from Theorem 3.3 that $\rho(T)<1$, where
$T=\sum_{s=1}^{r}E_sM_s^{-1}N_s$, i.e, the parallel multisplitting
block iterative method {\rm (\ref{r5})} converges to the unique
solution of {\rm(\ref{r1})} for any choice of the initial guess
$x^{(0)}$.
\end{proof}

\begin{theorem} Let $M$ be as in {\rm (\ref{wh1})}, {\rm (\ref{wh2})} and {\rm (\ref{wh2'})}.
For the splitting {\rm (\ref{yuan4})} of $M$, if\ $0\leq
\gamma_s\leq\omega_s\leq1$ and $0<\omega_s$ for $s=1,2,\cdots,t,$
then the parallel BGAOR iterative method {\rm (\ref{r8})} converges
to the unique solution of {\rm(\ref{r1})} for any choice of the
initial guess $x^{(0)}$.
\end{theorem}

\begin{proof}
The proof is similar to that for Theorem 4.1 and is easy to obtain
from Theorem 3.5.
\end{proof}

%%%%%%%%%%%%%%%%%%%%%%%%%%%%%%%%%%%%%%%

\section{Numerical examples} \label{examples-sec}
In this section some examples are given to illustrate the results
obtained in Section 3 and Section 4.

\begin{example} {\rm Let the coefficient matrix $A$ of linear system
(\ref{r1}) be given by} \end{example}
\begin{equation}\label{bs1}
 \ A=\left[
 \begin{array}{ccccccc}
 \ 3  & -2 & 2 & -1 & 1 & -1  \\
 \ -2  & 3 & -1 & 2 & -1 & 1 \\
 \ 40 & -35 & 100 & -80 & -50 & 40\\
 \ -35 & 40 & -80 &  90 & 40 & -40\\
 \ 3  & -3 & -6 & 4 & 10 & -8  \\
 \ -3  & 3 & 4 & -5 & -8 & 9
 \end{array}
 \right].
 \end{equation}
It is easy to see that $A\in H_3^2.$ Now we verify the convergence
results of some block iterative methods for linear systems with
given matrix $A\in H_3^2$ in Section 3.

We choose
\begin{equation}\label{bs2}
\begin{array}{lll}
 \ M_1&=&\left[
 \begin{array}{ccccccc}
 \ 3  & -2 & 2 & -1 & 1 & -1  \\
 \ -2  & 3 & -1 & 2 & -1 & 1 \\
 \ 0 & 0 & 100 & -80 & -50 & 40\\
 \ 0 & 0 & -80 &  90 & 40 & -40\\
 \ 0  & 0 & 0 & 0 & 10 & -8  \\
 \ 0  & 0 & 0 & 0 & -8 & 9
 \end{array}
 \right],
 \end{array}
 \end{equation}
 \begin{equation}\label{bs3}
 \ M_2=\left[
 \begin{array}{ccccccc}
 \ 3  & -2 & 0 & 0 & 0 & 0  \\
 \ -2  & 3 & 0 & 0 & 0 & 0 \\
 \ 40 & -35 & 100 & -80 & 0 & 0\\
 \ -35 & 40 & -80 &  90 & 0 & 0\\
 \ 3  & -3 & -6 & 4 & 10 & -8  \\
 \ -3  & 3 & 4 & -5 & -8 & 9
 \end{array}
 \right]
 \end{equation}
 and
 \begin{equation}\label{bs4}
 \ M_3=\left[
 \begin{array}{ccccccc}
 \ 3  & -2 & 0 & 0 & 0 & 0  \\
 \ -2  & 3 & 0 & 0 & 0 & 0 \\
 \ 0 & 0 & 100 & -80 & 0 & 0\\
 \ 0 & 0 & -80 &  90 & 0 & 0\\
 \ 0  & 0 & 0 & 0 & 10 & -8  \\
 \ 0  & 0 & 0 & 0 & -8 & 9
 \end{array}
 \right].
 \end{equation}
 Then, $N_s=M_s-A$ for $s=1,2,3.$
Set $E_1=diag(1/2,1/2,1/6,1/6,1/3,1/3)$,
$E_2=diag(1/3,1/3,1/2,1/2,1/6,1/6)$ and
$E_3=diag(1/6,1/6,1/3,1/3,1/2,1/2)$. Then, we have
$\sum_{s=1}^{3}E_s=I$, and consequently, $(M_s,N_s,E_s)_{s=1}^{3}$
is a multisplitting of the matrix $A$ and
$T=\sum_{s=1}^{3}E_sM_s^{-1}N_s$ is the iteration matrix. Direct
computation yields $\rho(T)=0.8987<1$, which shows that the parallel
multisplitting block iterative method {\rm (\ref{r5})} is
convergent.

\begin{example} {\rm Consider the
following linear system arising in the numerical solution of the
Euler equation \cite{{P.S8}}:}
\begin{equation}\label{raq1}
Mx=b, \end{equation}
\end{example}
 where $M\in
\complex^{(4\times 3\times 2)\times (4\times 3\times 2)}$ is as in
{\rm (\ref{wh1})}, {\rm (\ref{wh2})} and {\rm (\ref{wh2'})} and
$b=[1,3,1,2,5,3,2,1,7,5,9,$ $0,2,0,1,2,1,0,1,3,1.2,4,6,8]^T$. Here
$A^+=A^-=\left[
 \begin{array}{ccccccccc}
 \  2 &  -1 \\
 \  -1 & 2
 \end{array}
 \right]$, $B^+=\left[
 \begin{array}{ccccccccc}
 \  2 &  2 \\
 \  2 & 2
 \end{array}
 \right]$, $B^-=\left[
 \begin{array}{ccccccccc}
 \  2 &  -2 \\
 \  -2 & 2
 \end{array}
 \right]$ and $C=A^++A^-+B^++B^-=\left[
 \begin{array}{ccccccccc}
 \  8 &  -2 \\
 \  -2 & 8
 \end{array}
 \right]$. Then $A=A^+-A^-=0$ and
$B=B^+-B^-=\left[
 \begin{array}{ccccccccc}
 \  0 &  4 \\
 \  4 & 0
 \end{array}
 \right]$ and hence $N(A)\cap N(B)=\emptyset$. Then
 \begin{equation}\label{wh11}
M:=\left[
 \begin{array}{ccccccccc}
 \ T & S_1 &  &  \\
 \ S_2 & T & S_1 &  \\
 \   & S_2 & T & S_1\\
 \  &  & S_2 & T
 \end{array}
 \right]\in \complex^{(4\times 3\times 2)\times (4\times 3\times 2)},
\end{equation}
where $T,\ S_1,\ S_2\in \complex^{(3\times 2)\times (3\times 2)}$
are defined by
\begin{equation}\label{wh22}
\begin{array}{llll}
T&=&\left[
 \begin{array}{ccccccccc}
 \ C & -A^- &  &  \\
 \ -A^+ & C &  -A^- &  \\
 \  & -A^+ & C
 \end{array}
 \right],\\
S_1&=&\left[
 \begin{array}{ccccccccc}
 \  -B^- &  &  \\
 \   & -B^- &   \\
 \   &  & -B^-
 \end{array}
 \right],\ S_2=\left[
 \begin{array}{ccccccccc}
 \  -B^+ &  &  \\
 \   & -B^+ &   \\
 \   &  & -B^+
 \end{array}
 \right].
 \end{array}
\end{equation}
Writing $T=M_s-N_s$, where $M_s$ and $N_s$ are defined by
\begin{equation}\label{1wh22}
\begin{array}{llll}
M_1&=&\left[
 \begin{array}{ccccccccc}
 \ C & 0 &  &  \\
 \ -A^+ & C &  0 &  \\
 \  & -A^+ & C
 \end{array}
 \right],\ M_2=\left[
 \begin{array}{ccccccccc}
 \ C & -A^- &  &  \\
 \ 0 & C &  -A^- &  \\
 \  & 0 & C
 \end{array}
 \right]\\
M_3&=&\left[
 \begin{array}{ccccccccc}
  \ C & 0 &  &  \\
 \ -A^+ & C &  -A^- &  \\
 \  & 0 & C
 \end{array}
 \right],\ M_4=\left[
 \begin{array}{ccccccccc}
  \ C & -A^- &  &  \\
 \ 0 & C &  0 &  \\
 \  & -A^+ & C
 \end{array}
 \right]\\
 M_5&=&\left[
 \begin{array}{ccccccccc}
 \ C & -A^- &  &  \\
 \ -A^+ & C &  0 &  \\
 \  & 0 & C
 \end{array}
 \right],\ M_6=\left[
 \begin{array}{ccccccccc}
 \ C & 0 &  &  \\
 \ 0 & C &  -A^- &  \\
 \  & -A^+ & C
 \end{array}
 \right].
 \end{array}
\end{equation}
and $N_s=M_s-T$ for $s=1,2,3,4,5,6$, then we have a multisplitting
$(P_s,Q_s,E_s)_{s=1}^{r}$ of the matrix $M$ with $1\leq r\leq 6$,
where $P_s$ and $Q_s$ are defined by (\ref{yuan0}), (\ref{yuan1})
and (\ref{yuan2}), and $E_s=\frac{1}{r}I_{4\times 3\times 2}$, where
$I_{4\times 3\times 2}$ is the $(4\times 3\times 2)\times(4\times
3\times 2)$ identity matrix for $s=1,2,\cdots,r$. Furthermore, the
iteration matrix is $\mathbb{T}_{r}=\sum_{s=1}^{r}E_sP_s^{-1}Q_s$.
By direct computation, one obtains $\rho(\mathbb{T}_2)= 0.2901,\
\rho(\mathbb{T}_4)=0.2959,\ \rho(\mathbb{T}_5)=0.2894$ and
$\rho(\mathbb{T}_6)=0.2796$. This shows that the parallel
multisplitting block iterative method {\rm (\ref{r5})} for linear
system (\ref{raq1}) converges to the unique solution of
{\rm(\ref{raq1})} for any choice of the initial guess $x^{(0)}$.

In what follows we consider the convergence speed (i.e., quantity of
spectral radius of iteration matrix and number of iterations
required for given accuracy $\epsilon$) of the parallel
multisplitting method for different values of $r$. As is shown in
\cite{OL16} and \cite{white2}, for a given linear system, the
convergence speed of the parallel multisplitting method depends not
only on the choice of the parallel multisplitting of the coefficient
matrix and the weighting matrix but also on the number $r$ of
splittings in such a parallel multisplitting.

Tables 5.1-5.2 indicate the changing on both the quantity of
spectral radius of iteration matrix and the number $M$ of iterations
required for given accuracy
$\epsilon=\|x^{(M)}-x^{(M-1)}\|_2<10^{-4}$ for different $r$ and
different choice of weighting matrices $E_s$, where $\|x\|_2$
denotes $2-$norm of the vector $x$. The initial guess was taken to
be the vector of all one's.\\

{\small{
\begin{center}{\hbox{\bf Table 5.1.\ The comparison of convergence speed with different $r$ and $E_s=\frac{1}{r}I_{4\times 3\times 2}$}}  { {
\begin{tabular}{|c|c|c|c|c|c|c|c|}\hline{$r$}&
$1$& $2$ & $3$ &
$4$&$5$&$6$ \\
\hline $\rho(\mathbb{T}_{r})$ & $ 0.1801$ & $ 0.2901$ & $0.2844$&
$0.2959$ & $0.2894$ & $0.2796$ \\
\hline{Number of iterations} & $11$ & $13$ &
$13$ & $13$ & $13$ & $12$ \\
\hline\end{tabular}} }
\end{center}}}

~~~~~~~~~~~~~~~~~~~\\

{\small{
\begin{center}{\hbox{\bf Table 5.2.\ The comparison of convergence speed with different $r$ and $E_s$}}  { {
\begin{tabular}{|c|c|c|c|c|c|c|c|}\hline{$r$}&
$1$& $2$ & $3$ &
$4$&$5$&$6$ \\
\hline $\rho(\mathbb{T}_{r})$ & $ 0.1801$ & $ 0.1801$ & $0.1801$&
$0.1801$ & $0.2719$ & $0.2719$ \\
\hline{Number of iterations} & $11$ & $12$ &
$12$ & $12$ & $12$ & $12$ \\
\hline\end{tabular}} }
\end{center}}}
{\small Note that in table 5.2 the weighting matrices $E_s$ are
chosen as follows: $E_1=diag(I_6,0,I_6,0)$ and
$E_2=diag(0,I_6,0,I_6)$ when $r=2$; $E_1=diag(I_6,0,I_6,0)$,
$E_2=diag(0,I_6,0,0)$ and $E_3=diag(0,0,0,I_6)$ when $r=3$;
$E_1=diag(I_6,0,0,0)$, $E_2=diag(0,I_6,0,0)$, $E_3=diag(0,0,I_6,0)$
and $E_4=diag(0,0,0,I_6)$ when $r=4$; $E_1=diag(I_6,0,0,0)$,
$E_2=diag(0,I_6,0,0)$, $E_3=diag(0,0,I_6,0)$ and
$E_4=E_5=diag(0,0,0,\frac{1}{2}I_6)$ when $r=5$;
$E_1=diag(I_6,0,0,0)$, $E_2=diag(0,I_6,0,0)$,
$E_3=E_6=diag(0,0,\frac{1}{2}I_6,0)$ and
$E_4=E_5=diag(0,0,0,\frac{1}{2}I_6)$ when $r=6$, where $I_6$ is the
$6\times 6$ identity matrix.}

Finally, we test the convergence of the parallel BGAOR iterative
method {\rm (\ref{r8})} for linear system (\ref{raq1}).
 Assume that {\rm (\ref{wh11})} and {\rm (\ref{wh22})}
 hold. Let $M_s$ be defined as in {\rm (\ref{1wh22})} and $N_s=M_s-T$ for $s=1,2,3,4$.
Let $T=D'_s-L'_s-U'_s$, where $D'_s=M_s$, $L'_s=0$ and $U'_s=N_s$
for $s=1,2,3,4.$ Then $M=D_s-L_s-U_s$, where $D_s,\ L_s$ and $U_s$
are defined in {\rm (\ref{yuan5})} and {\rm (\ref{yuan6})}, and
thus, $(P_s,Q_s,E_s)_{s=1}^{4}$ is a multisplitting of the matrix
$M$, where $P_s={\omega}^{-1}(D_s-\gamma L_s),\
Q_s={\omega}^{-1}[(1-\omega)D_s+(\omega-\gamma)L_s+\omega U_s],
0\leq \gamma\leq\omega\leq1,\ 0<\omega$ and $E_s=0.25 I_{4\times
3\times 2}$ with $I_{4\times 3\times 2}$ the $(4\times 3\times
2)\times(4\times 3\times 2)$ identity matrix for $s=1,2,3,4$. As a
consequence, ${\cal{L}}(\gamma,\omega)=\sum_{s=1}^{r}E_sP_s^{-1}Q_s$
is the iteration matrix of the parallel BGAOR iterative method {\rm
(\ref{r8})}. Let $\rho(\cal{L}(\gamma,\omega))$ denote the spectral
radius of ${\cal{L}}(\gamma,\omega)$. The comparison results of
$\rho(\cal{L}(\gamma,\omega))$ with different parameter pairs
$(\gamma,\omega)$ are shown in Table 5.3 to show that the change of
the convergence of the parallel BGAOR iterative method with
parameter pair $(\gamma,\omega)$ changing.\\

{\small{
\begin{center}{\hbox{\bf Table 5.3.\ The comparison results of $\rho(\cal{L}(\gamma,\omega))$ with different
parameter pairs $(\gamma,\omega)$}}  { {
\begin{tabular}{|c|c|c|c|c|c|c|c|}\hline{$(\gamma,\omega)$}&
$(0.1,0.2)$& $(0.3,0.4)$ & $(0.5,0.6)$ &
$(0.7,0.8)$&$(0.8,0.9)$&$(0.9,1)$ \\
\hline $\rho(\cal{L}(\gamma,\omega))$ & $0.8592$ & $0.7184$ &
$0.5776$&
$0.4367$ & $0.3663$ & $0.2959$ \\
\hline{$(\gamma,\omega)$} & $(0.8,0.8)$ & $(0.9,0.9)$ &
$(0.9,0.95)$ & $(0.95,0.99)$ & $(0.99,0.99)$ & $(1,1)$ \\
\hline $\rho(\cal{L}(\gamma,\omega))$ & $0.4367$ & $0.3663$&
$0.3561$ & $0.3030$ & $0.3005$ & $0.2959$ \\
\hline\end{tabular}} }
\end{center}}}

The table shows that the change in the convergence of the parallel
BGAOR iterative method with change in the parameter pair
$(\gamma,\omega)$.

In the following, we will discuss the convergence of the parallel
BGAOR iterative method {\rm (\ref{r8})} for linear system
(\ref{raq1}). It is easy to see from Table 5.3 that
$\rho(\cal{L}(\gamma,\omega))$ decreases gradually when $r$ and
$\omega$ increase from $0.1$ and $0.2$, respectively, to $1$.
Furthermore, we have
\begin{equation}\label{ee1}\begin{array}{llll}\min\limits_{\gamma,\omega\in
(0,1],\gamma\leq\omega}\rho({\cal{L}}(\gamma,\omega))=\rho({\cal{L}}(
1,1))=\rho({\cal{L}}_{PBGGS}),
\end{array}\end{equation} where ${\cal{L}}_{PBGGS}$
denotes the iteration matrix of the parallel BGGS methods.

In addition, since the parallel BGSOR, the parallel BAOR and the
parallel BSOR methods are special cases of the parallel
BGAOR-method, the same results for the parallel BGSOR, the parallel
BAOR and the parallel BSOR methods can also obtained.

\begin{example} {\rm Consider a large sparse linear system arising in the numerical solution of the
elliptic equations \cite{R.S.14}:}
\begin{equation}\label{raq1c}
Ax=b, \end{equation}
\end{example}
where \begin{equation}\label{wh11c} A=\left[
 \begin{array}{ccccccccc}
 \ B & -I &  &  \\
 \ -I & B & \ddots &  \\
 \   & \ddots & \ddots & -I\\
 \  &  & -I & B
 \end{array}
 \right]\in \complex^{mn\times mn}
\end{equation}
where $I$ is the $m\times m$ identity matrix and $B\in
\complex^{m\times m}$ are defined by
\begin{equation}\label{wh11c1} B=\left[
 \begin{array}{ccccccccc}
 \ 4 & -1 &  &  \\
 \ -1 & 4 & \ddots &  \\
 \   & \ddots & \ddots & -1\\
 \  &  & -1 & 4
 \end{array}
 \right]\in \complex^{m\times m}.
\end{equation}

For $r=2$ and two positive integers $m_1,~m_2$ with $1\leq m_2<
m_1\leq n$, we define a multisplitting $A=D-L_s-U_s$ of the block
matrix $A$, where

\begin{equation}\label{1405a}
\begin{array}{llll}
D&=&diag[B,B,\cdots,B]\in \complex^{mn\times mn};\\
L_s&=&[L^{(s)}_{ij}]\in \complex^{mn\times mn},~~~s=1,2;\\
U_s&=&[U^{(s)}_{ij}]\in \complex^{mn\times mn},~~~s=1,2
 \end{array}
\end{equation} with
\begin{equation}\label{1405b}
\begin{array}{llllll}
L^{(1)}_{ij}&=& \left\{
\begin{array}{cc}
I,\ \ & \ \ \ \ j=i-1,~2\leq i\leq m_1,\\
0, \ &\ \ \ \ \ \ otherwise,
\end{array}\right.\\
L^{(2)}_{ij}&=& \left\{
\begin{array}{cc}
I,\ \ & \ \ \ \ \ j=i-1,~m_2\leq i\leq n,\\
0, \ \ & \ \ \ \ otherwis,
\end{array}\right.\\
U^{(1)}_{ij}&=& \left\{
\begin{array}{ccc}
I,\ \ & \ \ \ \ j=i-1,~m_1+1\leq i\leq n,\\
I,\ \ & \ \ \ \ j=i+1,~1\leq i\leq n-1,\\
 0, \ \ & \ \  \ otherwise,
\end{array}\right.\\
U^{(2)}_{ij}&=& \left\{
\begin{array}{ccc}
I,\ \ & \ \ \ \ j=i-1,~2\leq i\leq m_2-1,\\
I,\ \ & \ \ \ \ j=i+1,~1\leq i\leq n-1,\\
 0, \ \ & \ \  \ otherwise,
\end{array}\right.\\
\end{array}
\end{equation}
and two weighted matrices
\begin{equation}\label{1405c}
\begin{array}{llll}
E_s=diag[E^{(s)}_{11},\cdots,E^{(s)}_{nn}]\in \complex^{mn\times
mn},~~~s=1,2
 \end{array}
\end{equation}where
\begin{equation}\label{1405d}
\begin{array}{llllll}
E^{(1)}_{ii}&=& \left\{
\begin{array}{ccc}
I,\ \ & \ \ \ \ 1\leq i\leq m_2,\\
I/2,\ \ & \ \ \ \ m_2+1\leq i\leq m_1-1\\
 0, \ \ & \ \  \ m_1\leq i\leq n
\end{array}\right.\\
E^{(2)}_{ii}&=& \left\{
\begin{array}{ccc}
0,\ \ & \ \ \ \ 1\leq i\leq m_2,\\
I/2,\ \ & \ \ \ \ m_2+1\leq i\leq m_1-1,\\
 I, \ \ & \ \  \ m_1\leq i\leq n.
\end{array}\right.
\end{array}
\end{equation}

 We let (i)
$m_1=[\frac{3n}{4}],~m_2=[\frac{n}{4}]$; (ii)
$m_1=[\frac{5n}{6}],~m_2=[\frac{n}{6}]$, where [ ] denotes the
integer part of corresponding real number. Then we get two weighted
matrices $E_1$ and $E_2$.  The initial guess of $x_0$ is taken as a
zero vector. Here $\|x^{k+1}-x^{k}\|/\|x^{k+1}\|\leq 10^{-6}$ is
used as the stopping criterion. All experiments were executed on a
PC using MATLAB programming package.\\

{\small{
\begin{center}{\hbox{Table 5.4.\ Multisplitting BGAOR method with $n=m$}}  { {
\begin{tabular}{cccccccc}
\hline m & 5 & 7 & 11 & 13 & 15 & 20 \\
\hline (i) &  &  &  &  &  &  \\
 Time & $0.0483$ & $0.785$ & $0.892$&
$0.7120$ & $1.9663$ & $20.2959$ \\
Iter  & 19 & 30 &
56 & 75 & 93 & 148 \\
\hline (ii) &  &  &  &  &  &  \\
 Time & $0.0613$ & $0.0837$ & $0.0880$&
$0.7052$ & $1.9551$ & $20.3108$ \\
Iter  & 19 & 30 &
56 & 75 & 93 & 148 \\
\hline\end{tabular}} }
\end{center}}}
~~~~~~~~~~~~~~\\

{\small{
\begin{center}{\hbox{Table 5.5.\ Multisplitting BGAOR method when the case (i) and (ii) for $n=m=10$.}}  { {
\begin{tabular}{cccccccc}
\hline ($\gamma,\omega$) & (0.9,1) & (0.7,1) & (0.5,1) & (0.7,1.1) & (1.1,1) & (1,1) \\
\hline (i) &  &  &  &  &  &  \\
 Time & $0.0753$ & $0.0815$ & $0.0819$&
$0.0895$ & $0.0884$ & $0.0726$ \\
Iter  & 42 & 51 &
52 & 105 & 84 & 39 \\
\hline (ii) &  &  &  &  &  &  \\
 Time & $0.1130$ & $0.0737$ & $0.0810$&
$0.103$ & $0.0923$ & $0.0731$ \\
Iter  & 44 & 51 &
56 & 115 & 83 & 41 \\
\hline\end{tabular}} }
\end{center}}}

In Table 5.4, $\gamma=\gamma_1=\gamma_2=0.7$ and
$\omega=\omega_1=\omega_2=1$, we report the CPU time (Time) and the
number of iterations (Iter) for the multisplitting block GAOR
iterative method. In Tables 5.5, let $m=10$, we report the CPU time
(Time) and the number of iterations (Iter) for the multisplitting
block GAOR iterative method for different $\gamma$ and $\omega$.
Following from Tables 5.5, for ($\gamma,\omega$)=(1,1) it can be
seen that the convergence rate of the multisplitting block GAOR
iterative method is faster the other parameterized iterative method
for generalized $H-$matrices.

%%%%%%%%%%%%%%%%%%%%%%%%%%%%%%%%%%%%%%%

\section{Conclusions} \label{conclusions-sec}

The paper is devoted to the study of the convergence properties of
some parallel multisplitting block iterative methods for the
solution of linear systems arising in the numerical solution of the
Euler equation. We give sufficient conditions for the convergence of
parallel multisplitting block iterative methods including the
parallel block generalized AOR (BGAOR), the parallel block AOR
(BAOR), the parallel block generalized SOR (BGSOR), the parallel
block SOR (BSOR), the extrapolated parallel BAOR and the
extrapolated parallel BSOR methods. Furthermore, we present the
convergence of the parallel block iterative methods for linear
systems with special block tridiagonal matrices arising in the
numerical solution of the Euler equation. Finally, we have given
some examples to demonstrate the convergence results obtained in
this paper.

% There are various ways to include postscript figures in a
% LaTeX file. Consult the manual for details.
% Below is a particular example with an encapsulated postscript file
% Note that the epsfig package must be included in the preamble.

% \begin{figure}[ht]
% \begin{center}
% \epsfig{file=circle.eps,height=38mm,width=38mm,clip=}
% \caption{A circle!}
% \end{center}
% \label{fig}
% \end{figure}

{\bf Acknowledgments.} The first author would like to thank
Professor Michele Benzi for his help. The authors are grateful to
the referees for their valuable suggestions.

%%%%%%%%%%%%%%%%%%%%%%%%%%%%%%%%%%%%%%%%%%%%%%%%%%%%%%%%%%%%%

\end{document}